\documentclass{amsart}
\usepackage{amsthm,graphicx,amssymb}
\usepackage[usenames]{color} %color
\usepackage{colortbl}
\usepackage{caption}
\usepackage{subcaption}
\usepackage{hyperref}

\usepackage{lineno}
\usepackage{tikz}
\usepackage{tikz-qtree}
\usetikzlibrary{arrows}
\usetikzlibrary{calc}
\theoremstyle{definition}
 \newtheorem{theorem}{Theorem}
 \newtheorem{proposition}{Proposition}
 \newtheorem{definition}{Definition}
 
 \newtheorem{lemma}{Lemma}

\newcommand\ZZ{\mathbb Z}
\newcommand\NN{\mathbb N}
\newcommand\G{\Gamma}

\begin{document}

\title{Sandpiles on the heptagonal tiling}

\author[N. Kalinin, M. Shkolnikov]{Nikita Kalinin, Mikhail Shkolnikov}

\thanks{
Research is supported in part the grant 159240 of the Swiss National Science
Foundation as well as by the National Center of Competence in Research
SwissMAP of the Swiss National Science Foundation.  
}

\address{Universit\'e de Gen\`eve, Section de
  Math\'ematiques, Route de Drize 7, Villa Battelle, 1227 Carouge, Switzerland}

\keywords{self-organized criticality, sandpiles, hyperbolic tilling}
\email{Nikita.Kalinin\{at\}unige.ch \hfil\break mikhail.shkolnikov\{at\}gmail.com}

\maketitle

%${\EuScript{A} include amoebas}$
\begin{abstract}
We study perturbations of the maximal stable state in a sandpile model on the set of faces of the heptagonal tiling on the hyperbolic plane. An explicit description for relaxations of such states is given.
\end{abstract}
{\fontfamily{calligra}\selectfont \begin{small}  Dedicated to Sergei Vassilyevich Duzhin, who was a great lector and teacher -- that is why this article is about sandpiles, one of his last courses, and passionate experimenter, who loved concrete calculations -- that is why this article is about experimental results. \end{small}}

\section{Introduction}

 \begin{figure}[h]
    \centering
    \setlength\fboxsep{0pt}
    \setlength\fboxrule{0.5pt}
    \includegraphics[width=0.6\textwidth]{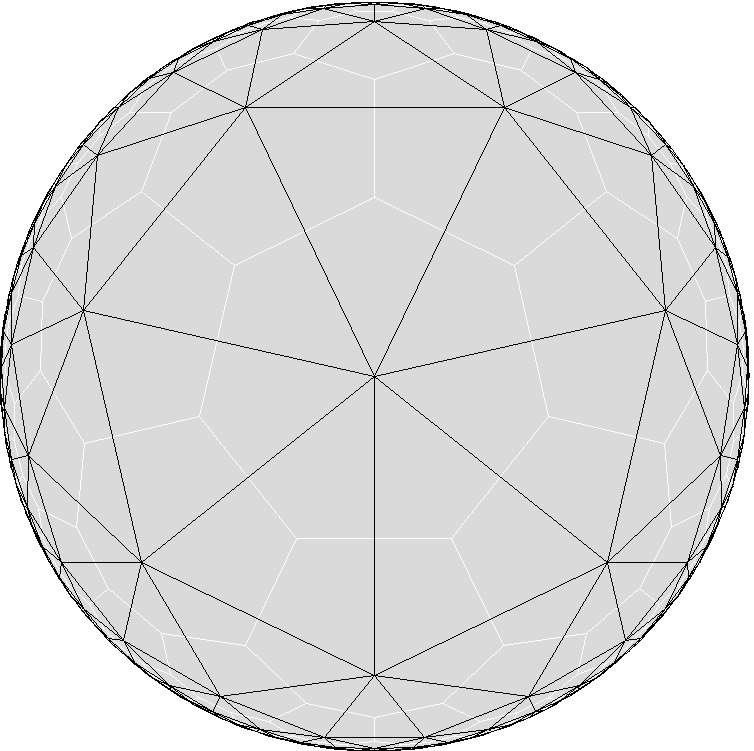}\ 

    \caption{The triangular tiling (black edges) and the dual heptagonal tiling (white edges). The hyperbolic plane is represented by the Klein disk model.} 
    
    \label{fig_tiling}
\end{figure}

Consider a M\"obius triangle $\Delta\subset \mathbb{H}^2$ whose angles are all  equal ${2\pi\over 7}.$ Note that such a triangle is unique up to an isometry of $\mathbb{H}^2$. Let $G$ be the subgroup of $\mathop{Aut}(\mathbb{H}^2)$ generated by three reflections with respect to the sides of $\Delta.$ Action of $G$ on $\Delta$ produces the regular triangular tiling of $\mathbb{H}^2$ (see \cite{MagnusTriangle} and Figure \ref{fig_tiling}). Consider a graph $\Gamma$ whose set of vertices (edges) consists of translations by $G$ of the vertices (edges) of $\Delta.$  Note that $\Gamma$ is a regular graph, all its vertices have degree $7$.

The dual graph of $\Gamma$ can be embedded to $\mathbb{H}^2$ by sending its vertices to the centers of the corresponding triangles and connecting the centers of adjacent triangles with a geodesic. This procedure defines the regular heptagonal tiling. We think of this tiling as the closest hyperbolic analog of the hexagonal honeycomb tiling of Euclidean plane.
    
Choose a vertex $O$ of the graph $\Gamma.$ For any vertex of $\Gamma$ denote by $L(v)$ the distance from $v$ to $O$ in $\Gamma,$ i.e. $L(v)$ is the minimal length of a chain of edges connecting $O$ and $v.$ For $m\in\NN$ denote by $\Gamma_m$ the set of all vertices $v$ in $\Gamma$ such that $L(v)\leq m.$  

For any $v\in\Gamma_m$ consider {\it a toppling} operator $T_v$ on the set of all integer-valued functions on $\Gamma_m,$ i.e. $$T_v\phi (u)=\phi(u)+\Delta\delta_v (u),$$ for $\phi\colon \Gamma_m\rightarrow \ZZ$ and $u\in\Gamma_m.$ Here we denote by $n(v)$ the set of all $7$ neighbors of $v$ in $\Gamma$ and $\delta_v$ is a function on $\Gamma$ sending $v$ to $1$ and vanishing at all other vertices. The discrete Laplacian operator $\Delta$ is given by $$\Delta f(v)=-7 f(v)+\sum_{u\in n(v)\cap \Gamma_m} f(u),$$ where $f$ is a function on $\Gamma.$

A {\it state} of the sandpile model on $\Gamma_m$ is a non-negative integer-valued function on $\Gamma_m.$ We interpret a state as a distribution of sand grains (or chips) on $\Gamma_m.$ A toppling $T_v$ is called {\it legal} for a state $\phi$ if $T_v\phi$ is a state, in other words $\phi(v)\geq 7.$ A state $\phi$ is called {\it stable} if there are no legal topplings for $\phi,$ i.e. $\phi(v)<7$ for any $v\in\Gamma_m.$ 

Consider a state $\psi_0.$ A sequence of states $\psi_0,\dots,\psi_N$ is called {\it a relaxation} of $\psi_0$ if $\psi_N$ is a stable state and for $i=0,\dots,N-1$ the state $\psi_{i+1}$ is the result of applying a legal toppling to $\psi_i.$ It is a basic statement (see \cite{Dhar,LP}) in general theory of abelian sandpiles that for any state $\psi_0$ there exist a relaxation and the resulting stable state $\psi_N$ depends only on $\psi_0$ and doesn't depend on a particular choice of a relaxation. Thus, we denote by $\psi_0^\circ$ the resulting state of any relaxation sequence of $\psi_0.$

Denote by $\mathcal{T}_{\psi_0}$ a function on $\Gamma_m$ counting the number of topplings at $v$ in a relaxation of $\psi_0.$ This function is called the {\it toppling function} (or odometer) of $\psi_0$ and for a given relaxation sequence $\psi_0,\dots,\psi_N$ can be expressed as $$\mathcal{T}_{\psi_0}=\sum_{i=1}^N\delta_{v_i},$$ where $\psi_{i+1}=T_{v_i}\psi_i$ and $v_i\in\Gamma_m.$ The toppling function doesn't depend on a choice of a relaxation sequence (see \cite{Dhar,FLP}). Clearly, the result of the relaxation for $\psi$ can be expressed in terms of toppling function as $\psi^\circ=\psi+\Delta T_\psi.$

The first version of a sandpile model was introduced in \cite{BTW} where a piece of a standard square lattice was used instead of $\Gamma_m.$ The model was generalized to arbitrary graphs in \cite{Dhar}, the generalization essentially coincides with {the chip-firing game}, well known in combinatorics  (see \cite{Mer}). The sandpile model was extensively studied in the case of embedded graphs including the classical square lattice and other tilings of Euclidean plane (see for example \cite{ATDNMAR, BTW, CPS, CPS2, Dhar,FLP,BR,LPS,LP,PS, PP}). To our knowledge, sandpiles on hyperbolic graphs has never been described in the literature. Our goal is to find the analogues of our recent results (see \cite{us}) about sandpiles on Euclidean lattices in the case of a hyperbolic tilling. The results established in this paper demonstrate that the study of sandpiles on hyperbolic graphs in general might be easier than on Euclidean ones. In Section \ref{discussion} we discuss possible generalizations.

\section{Statement of the results}
\begin{definition}
\label{def_maximal}
The {\it maximal stable state} $\phi_m$ on $\Gamma_m$ is the state given by $\phi_m(v)=6$ for all $v\in\Gamma_m.$ For a non-empty subset $P$ of $\Gamma_m$ consider a perturbation $\phi_m^P$ of the maximal stable state given by $$ \phi_m^P=\phi_m+\sum_{v\in P}\delta_v.$$  
\end{definition}

 \begin{figure}[b]
    \centering
    \setlength\fboxsep{0pt}
    \setlength\fboxrule{0.5pt}
    \includegraphics[width=0.8\textwidth]{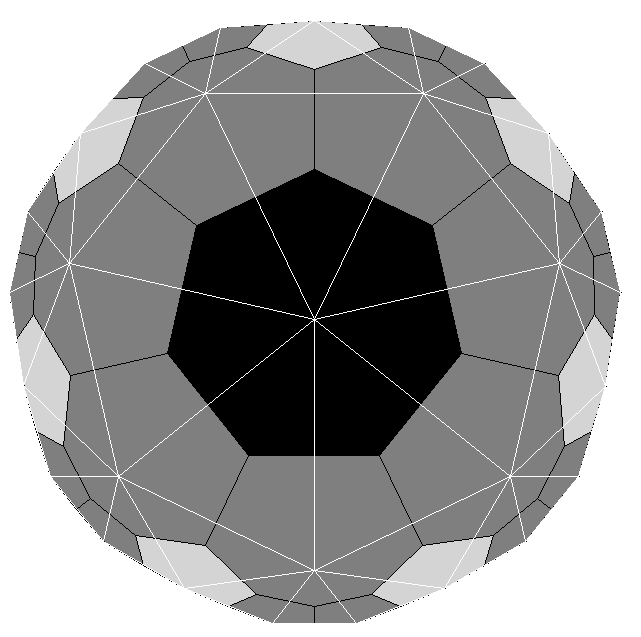}\ 

    \caption{The state $\beta_m^P$ for $m=2$ and $P=\{O\}.$  The black cell is dual to the vertex $O.$ The colors on the dual tiling represent numbers of sand at each vertex of $\Gamma_2$. The black cell represents $0,$ dark and light grey cells represent $3$ and $5$ sand grains respectively.} 
    
    \label{fig_depth21}
\end{figure}

We will describe the result of the relaxation $\beta_m^P=(\phi_m^P)^\circ$ of $\phi_m^P$ and the toppling function $\mathcal{T}_{\phi_m^P}.$ Consider a special case of the problem when $P=\{O\}.$ It is easy visible that the states $\beta_2^{P}$ (see Figure \ref{fig_depth21}) and  $\beta_6^{P}$ (left side of the \ref{fig_depth61}) coincide on $\Gamma_2$. We can also observe self-similar branches on the right side of Figure \ref{fig_depth61}. Though it is hard to express this fractality symbolically, we will give an easy combinatorial recipe for constructing the state $\beta_m^P$ without performing the actual relaxation (this is done in Proposition \ref{prop_thestate} and Definition \ref{def_univ}). 

 \begin{figure}
    \centering
    \setlength\fboxsep{0pt}
    \setlength\fboxrule{0.5pt}
    \includegraphics[width=0.45\textwidth]{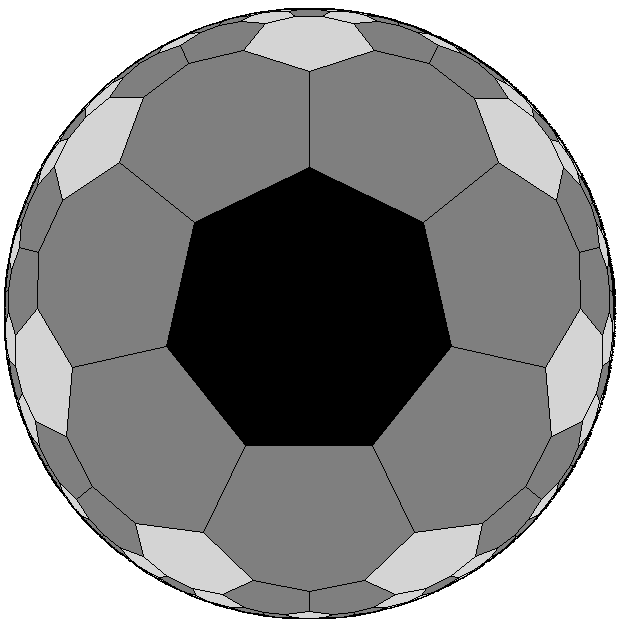}\ 
    \includegraphics[width=0.45\textwidth]{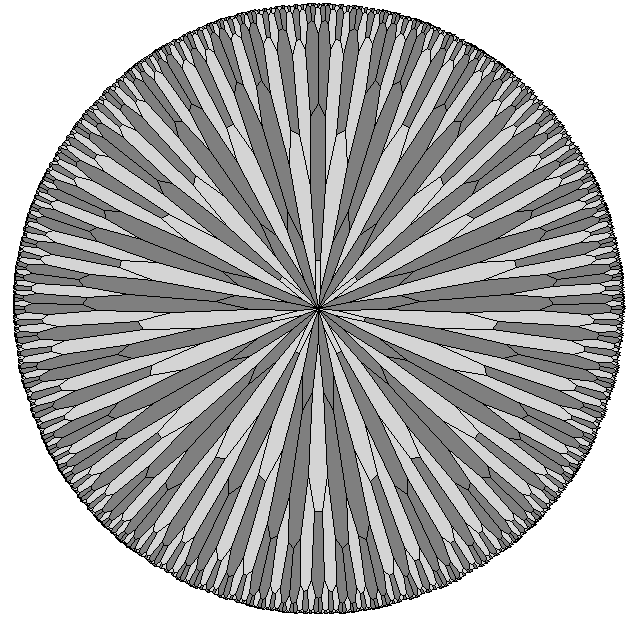}\ 
    \caption{Two pictures of the state $\beta_6^{\{O\}}:$ in the standard Klein model and after a homothety with ratio $0.005.$ } 
    
    \label{fig_depth61}
\end{figure}  

\begin{definition}For a state $\phi$ denote by $M\phi$ the {\it mass} of the state given by $$M\phi=\sum_{v\in \Gamma_m}\phi(v).$$
\end{definition}

The quantity $M\phi-M\phi^\circ$ represents the total number of sand grains that leave the system during the relaxation of $\phi.$ It appears that the state $\phi_m^P$ looses a substantial amount of sand after its relaxation and this amount doesn't depend on the configuration of points $P.$ 

\begin{proposition}\label{prop_mass}
For all $m>0$ there exist a constant $C_m>0$ such that $$C_m=M\phi_m^P-M\beta_m^P$$ is the same for all non-empty $P\subset\Gamma_m.$ Moreover, $$\lim_{m\to\infty}{ {C_m}\over{|\Gamma_m|}}=\sqrt{5}.$$
\end{proposition}

See page \pageref{proof_prop_mass} for the proof. In particular, we see that in average a element of $\Gamma_m$ looses more than two grains during the relaxation even if $|P|=1$. Thus, adding to $P$ more points produces no unstable cites and does not change the picture outside of the added points. This property conceptually explains  the existence of a universal ``solution'' producing all the states $\beta_m^P.$ 

Recall that $L(v)$ denotes the distance from $v$ to $O$ in $\Gamma$. We define $L(P)=\min_{p\in P}L(p).$ 
\begin{proposition}\label{prop_thestate}
There exist a family $\alpha_0,\alpha_1,\alpha_2,\dots$ of integer-valued functions on the set of vertices of $\Gamma$ with the following properties 
\begin{itemize}
\item $\alpha_0$ and $\alpha_s$ coincide outside $\Gamma_s;$
\item $\alpha_{s}$ is equal to $6$ on $\Gamma_{s-1};$
\item  for any $m\in\NN,$ non-empty $P\subset\Gamma_m$ and $s=L(P)$ \begin{equation*}\begin{aligned}\beta^P_m(v)= &\alpha_{s}(v),\text{ for } v\in\Gamma_m\backslash P,\\
\beta^P_m(p)=& \alpha_{s}(p)+1,\text{ for } p\in P. 
\end{aligned}
\end{equation*}
\end{itemize}

\end{proposition}

Proposition \ref{prop_thestate} is a consequence of the following 

\begin{theorem}\label{th_thefunction}
For any non-empty $P\subset\Gamma_m$ 
$$\mathcal{T}_{\phi_m^P}(v)=\min(m+1-L(v),m+1-L(P)).$$
\end{theorem}

In other words, the toppling function essentially doesn't depend on the number and the position of points $P$ and depends only on the distance from $O$ to $P.$ Theorem \ref{th_thefunction} and Proposition \ref{prop_thestate} are proved on the page \pageref{proof_th_thefunction}.

\section{Combinatorics of $\Gamma$}
The set of vertices of $\Gamma$ is a disjoint union of the sets $\Gamma_m^\bullet=\Gamma_m\backslash\Gamma_{m-1},$  $m\in\mathbb{Z}_{\geq 0}.$  We call $\Gamma_m^\bullet$ the $m$-th level of $\Gamma.$ For any pair of integers $(m,k)$ we consider the set $\Gamma_m^k$ given by $$\Gamma_m^k=\{v\in\Gamma_m^\bullet: |n(v)\cap\Gamma_{m-1}^\bullet|=k\}.$$ It is easy to see that $\Gamma_0=\Gamma_0^0=\{O\}$ and $\Gamma_m^\bullet=\Gamma_m^1\cup\Gamma_m^2$ for $m>0.$

 \begin{figure}[h]
 \resizebox{0.95\textwidth}{!}{
\begin{tikzpicture}[>=stealth',shorten >=1pt,auto,node distance=3cm,
  thick,tnode/.style={circle,fill=white!20,draw,font=\sffamily}]

  \node[tnode] (1) {1};
  \node[tnode] (2) [left of=1] {};

    \node[tnode] (3) [right of=1] {};
     \node[tnode] (4) [below of=1] {};
   \node[tnode] (8) [above left of=3] {1};
      \node[tnode] (5) [above right of=1] {2};
        \node[tnode] (6) [above right of=2] {1};
         \node[tnode] (7) [above left of=1] {2};
         
         \node[tnode] (10)[right of=3] {};
         
         \node[tnode] (11)[right of=10] {2};
  \node[tnode] (12) [left of=11] {};
    \node[tnode] (13) [right of=11] {};
     \node[tnode] (14) [below of=12] {};
      \node[tnode] (19) [below of=13] {};

      \node[tnode] (15) [above right of=11] {2};
    
         \node[tnode] (17) [above left of=11] {2};
    
        \node[tnode] (16) at ($(15)!0.5!(17)$) {1};

  \path[every node/.style={font=\sffamily\small}]
    (1) edge (4)
        edge (2)
        edge (3)
        edge (5)
        edge (6)
        edge (7)
        edge (8)
    (4) edge (2)
        edge (3)
    (4) edge (2)
    (7) edge (2)
        edge (6)
    (5) edge (3)
        edge (8)
    (6) edge (8)
   
   (11) edge (14)
        edge (12)
        edge (13)
        edge (15)
        edge (17)
   (14) edge (12)
    (14) edge (19)
    (19) edge (11)
        edge (13)
    (17) edge (12)
    (15) edge (13)
    (16) edge (15)
         edge (17)
         edge (11);
    
   \path[draw, dotted]
    (17) edge node {$\Gamma_{m+1}^\bullet$} (5)
    (3) edge node {$\Gamma_{m}^\bullet$} (12)
    (4) edge node {$\Gamma_{m-1}^\bullet$} (14);

\end{tikzpicture}}
  \caption{Vertices of first and second types. A type of a vertex is written in the circle, if it can be determined using the picture.} 
    \label{fig_types1}
    
\end{figure}
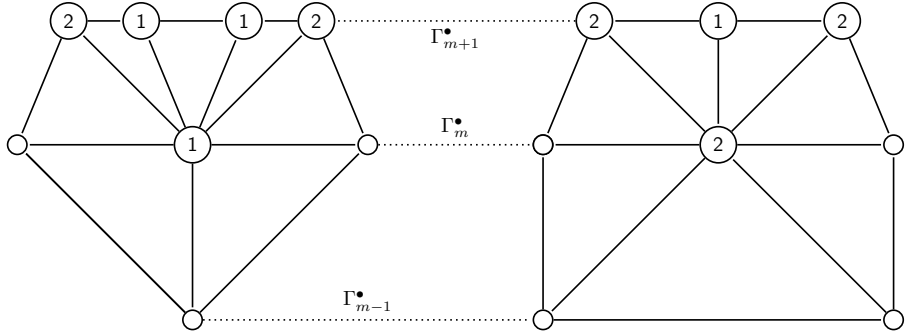

We distinguish three types of vertices. The {\it zeroth} type consist of unique vertex $O.$ The {\it first} type consists of vertices connected by a single edge with a previous level. And a vertex of the {\it second} type is connected with the previous level by two edges (see Figure \ref{fig_types1}).

We are going to compute the numbers $a_m=|\Gamma_m^1|$ and $b_m=|\Gamma_m^2|$ of vertices of the first and second types at $m$-th level. It is clear that $b_1=0,$ $a_1=7,$ 
\begin{equation*}
\begin{aligned}
b_{m+1}=& a_m+b_m \text{\ \ and }\\
a_{m+1}=& 2a_m+b_m=b_{m+1}+a_m.
\end{aligned}
 \end{equation*}
 Therefore, the sequence $b_1,a_1,b_2,a_2,b_3,a_3,\dots$ coincides with the Fibonacci sequence multiplied by $7.$ We have proven the following

\begin{lemma}\label{lem_fibab}
For any $m>0$ $$b_m=7u_{2m-2}\text{\ and\ }a_m=7u_{2m-1},$$ where $u_0=0,u_1=1$ and $u_{m+2}=u_{m+1}+u_m.$
\end{lemma}

Now we can compute the number of elements in $\Gamma_m.$

\begin{lemma}\label{lem_total}
$|\Gamma_m|=7u_{2m+1}-6.$
\end{lemma}
  
\begin{proof}
Indeed, the number $|\Gamma_m|$ is equal to $$1+\sum_{i=1}^m(a_i+b_i)=1+\sum_{i=1}^mb_{i+1}=1+\sum_{i=1}^m(a_{i+1}-a_i)=1+a_{m+1}-a_1.$$  
\end{proof}

\section{Wave action}

Consider a vertex $p\in\Gamma_m.$ Let $\phi$ be a state on $\Gamma_m.$ The ``wave action'' $W_p\phi$ at $p$ on $\phi$ is defined as follows. 
\begin{definition}\label{def_wave}
If $\phi(p)=6$ and there exists $v\in\Gamma_m$ such that $v$ is adjacent to $p$ and $\phi(v)=6$ then
$W_p\phi=(T_vT_p\phi)^\circ.$ Otherwise, we define $W_p\phi$ to be $\phi.$
\end{definition}

It is easy to show that the wave action doesn't depend on the choice of $v$ adjacent to $p.$ Clearly, the states $W_v\phi$ and $W_p\phi$ coincide. Therefore, for a given chain of vertices $p_1,\dots,p_N\in\Gamma_m$ such that $\phi(p_i)=6$ and $p_{i+1}\in n(p_i)$ all the states $W_{p_i}\phi$ coincide.

One can use a sequence of waves to decompose the relaxation of the state $\phi+\delta_p.$ Indeed, one can verify that $(\phi+\delta_p)^\circ=(W_p\phi(v)+\delta_p)^\circ.$ Therefore, for a sufficiently large $N$ either $W^N_p\phi(p)=W_p(W_p(\dots \phi)\dots)(p)<6$ and  $$(\phi+\delta_p)^\circ=W^N_p\phi(p)+\delta_p$$ or $W^N_p\phi(p)=6$ and $W^N_p\phi(v)<6$ for any $v\in n(p)\cap\Gamma_m$ and  $$(\phi+\delta_p)^\circ=T_p(W^N_p\phi(p)+\delta_p).$$

For more details about waves see \cite{us}.

 \begin{figure}[t]
    \centering
    \setlength\fboxsep{0pt}
    \setlength\fboxrule{0.5pt}
    \includegraphics[width=\textwidth]{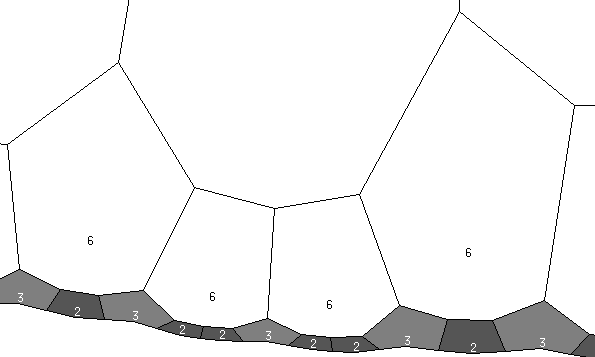}\ 

    \caption{A plot of the state $W_O(\phi_m)$ (see Definitions \ref{def_maximal},\ref{def_wave}) near the boundary of $\Gamma_5.$} 
    
    \label{fig_wave1}
\end{figure}  

Now we want to describe the wave action at $O$ applied to the maximal stable state $\phi_m.$ Since every point $v\in\Gamma_m$ can be connected with $O$ by a chain of adjacent cites and all the cites are maximal for $\phi_m,$ the state $W_O\phi_m$ is the result of applying a toppling to $\phi_m$ at every point of $\Gamma_m.$ By counting the incoming and outgoing grains at each vertex we conclude that $W_O\phi_m$ is equal to $6$ on $\Gamma_{m-1},$ to $2$ on $\Gamma_m^1$ and to $3$ on $\Gamma_m^2$ (see Figure \ref{fig_wave1}).

  The main observation of this paper is that the result of applying of the second wave coincides with the result of applying a single wave on a smaller domain, i.e. $$W_O^2\phi_m(v)=W_O\phi_{m-1}(v)\ \text{ for any }v\in\Gamma_{m-1}.$$  Indeed, applying a toppling to $W_{O}\phi_m$ at every vertex of $\Gamma_{m-1}$ we get a stable state which is equal to $6$ on $\Gamma_{m-2},$ to $2$ on $\Gamma_{m-1}^1,$ to $3$ on $\Gamma_{m-1}^2\cup\Gamma_{m}^1$ and to $5$ on $\Gamma_m^2$ (see Figure \ref{fig_wave2}).

This motivates the definition of the universal family $\alpha_s$. 
\begin{definition}\label{def_univ}
If $s>0$ then $\alpha_s$ is equal to 
\begin{itemize}
\item $6$ on $\Gamma_{s-1},$ 
\item $5$ on $\Gamma^2_m$ for $m>s,$ 
\item $3$ on $\Gamma^2_s$ and $\Gamma^1_m$ for $m>s$ and 
\item $2$ on $\Gamma^1_s.$
\end{itemize} 
The function $\alpha_0$ is defined to be $5$ on $\Gamma^2_m$ and $3$ on $\Gamma^1_m$ for $m>0$ and $-1$ on $\Gamma_{0}.$  
\end{definition}

 \begin{figure}[t]
    \centering
    \setlength\fboxsep{0pt}
    \setlength\fboxrule{0.5pt}
    \includegraphics[width=\textwidth]{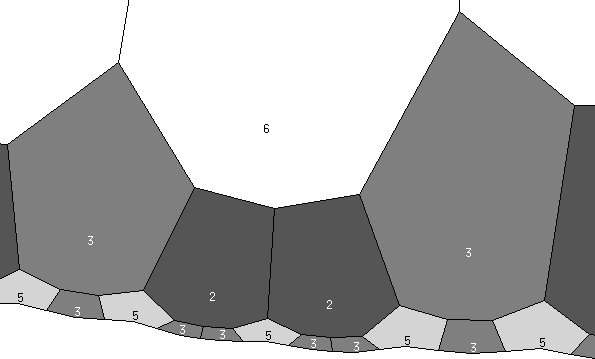}\ 

    \caption{A plot of the state $W^2_O\phi_m$ near the same piece of the boundary as on Figure \ref{fig_wave1}.} 
    
    \label{fig_wave2}
\end{figure}

With this definition and the complete description of the wave action we deduce the following 
\begin{lemma}\label{lem_main}
Consider a point $p\in\Gamma_0.$ Then $\beta^{\{p\}}_m=\alpha_{L(p)}+\delta_p$ and $$\mathcal{T}_{\phi_m^{\{p\}}}(v)=\min(m+1-L(p),m+1-L(v)).$$
\end{lemma}

\begin{proof}
We decompose the relaxation process for $\phi_m+\delta_p$ into a sequence of waves actions applied to the maximal stable state. From the structure of the wave action it follows that the state $W_0^k\phi_m$ coincides with $\alpha_{m-k+1}$ restricted to $\Gamma_m.$ In particular,  $W_0^k\phi_m$  is maximal on $\Gamma_{m-k}$ and is equal to $2$ or $3$ on $\Gamma^\bullet_{m-k+1}=\Gamma^\bullet_{m-k+1}\backslash \Gamma^\bullet_{m-k}.$ Indeed, applying a toppling to $W_0^k\phi_m$ at each point of $\Gamma_{m-k}$ we get a stable state which coincides with $W_0^{k+1}\phi_m.$ 

A point $p$ is maximal for $W_0^k\phi_m$ if $k<m-L(p).$ This implies that $W_0^{k+1}\phi_m=W_p^{k+1}\phi_m.$ Using the explicit description of the wave action, we see that $p$ is not maximal for $W_p^{m-L(p)}\phi_m$ and $\alpha_{L(p)}+\delta_p$ coincides with $\beta_m^{\{p\}}$ on $\Gamma_m.$ Combining these observations, we conclude that the toppling function for $\alpha_m+\delta_p$ is equal to $m+1-L(p)$ on $\Gamma_{L(p)}$ and $m-L(p)-k$ on $\Gamma^\bullet_{L(p)+k}.$
\end{proof} 

\begin{proof}[Proof of Theorem \ref{th_thefunction} and Proposition \ref{prop_thestate}] \label{proof_th_thefunction} Consider a configuration of vertices $P\subset \Gamma_m.$ Let $p_0\in P$ be a minimizer of $$l=\min_{p\in P} L(p),$$ i.e. $L(p_0)=l.$ 
By Lemma \ref{lem_main} the state $(\phi_m+\delta_{p_0})^\circ$ equals to $6$ exactly on $\Gamma_{l-1}$ and therefore the state $$(\phi_m+\delta_{p_0})^\circ+\sum_{p\in P,p\neq p_0}{\delta_p}$$ is stable and is equal to $\beta^P_m.$ 

\end{proof}

In particular, we see that some amount of sand is lost only after sending a first wave. We can compute the exact loss using Lemma \ref{lem_fibab}. A vertex of the first type on the boundary looses $4$ grains of sand and a vertex of the second type looses $3$ grains. Therefore, the total loss is equal to $4a_m+3b_m=7(4u_{2m-1}+3u_{2m-2})=7(2u_{2m}+u_{2m+1})=7(u_{2m}+u_{2m+2}).$

\begin{proof}[Proof of Proposition \ref{prop_mass}]\label{proof_prop_mass} It is well known that $$u_m\sim {1\over{\sqrt 5}} \rho^m\text{, where\ }\rho={1\over 2}(1+\sqrt{5}).$$
Then applying Lemma \ref{lem_total}
$${ {M\phi_m^P-M\psi_m^P}\over{|\Gamma_m|}}\sim{7{\rho^{2m}(1+\rho^2)}\over {7\rho^{2m+1}}}=\rho^{-1}+\rho=\sqrt{5}.$$
\end{proof}

\section{Discussion}\label{discussion}
The problem that we studied in this paper (sandpile on a ``hyperbolic lattice'') is similar to the one discussed in \cite{us} (sandpile on the two dimensional Euclidean lattice). For Euclidean case, the basic observation is that {\it the amount of lost sand} is {\it small} and therefore the resulting state is close to the maximal stable state. In fact, the Euclidean analogs of the states $\beta_m^P$ (see Definition~\ref{def_maximal}) coincide with the maximal stable state {\it almost everywhere.} In the {\it scaling limit}, the locus of vertices where such a state deviates from the maximum, looks like a thin graph (or more precisely, a tropical curve) passing through the perturbation points. In particular, the result of the relaxation strongly depends on the position of points $P.$ Surprisingly, the situation is very different in the hyperbolic case.

The key ingredient in \cite{us} is the description of the discrete harmonic functions with sublinear growth on large subsets of the lattice $\mathbb Z^2$. Also, instead of considering only sets like $\Gamma_n$ we could consider polygonal subsets of $\mathbb Z^2$, it was important that the boundary of such subsets if defined by sets of zeroes of discrete harmonic functions. We do not know much about discrete harmonic functions on $\Gamma$, therefore we had only one type of subsets of $\Gamma$ to consider. 

\subsection{Scalings}\label{sec_scaling}
 Here we propose a possible procedure to obtain a scaling limit in the hyperbolic case. Consider a family of graphs $\Gamma[n],$ $n\in\mathbb{N}$ embedded to $\mathbb{H}^2$ and satisfying the following properties. The graph $\Gamma=\Gamma[1]$ is isohedral (i.e. tile-transitive, all the faces are congruent). The graph $\Gamma[n+1]$ is a result of a finite subdivision of each cell of $\Gamma[n],$ moreover we ask this subdivision to be invariant with respect to the group of symmetries of $\Gamma[n].$ The set $V(\Gamma[n])$ of vertices of $\Gamma[n]$ is dense in $\mathbb{H}^2$ when $n$ goes to infinity, i.e. $\lim_{n\rightarrow \infty}\mathrm{dist}(V(\Gamma[n]),p)=0,$ for any $p\in\mathbb{H}^2.$ 

Consider a geodesic-convex set $\Omega\subset\mathbb{H}^2$ and collection of points $P\subset\Omega^\circ.$ For simplicity assume that $P\subset\Gamma[N]$ for some $N>0.$ Let $\phi_m$ be the maximal stable state on $V(\Gamma[m])\cap\Omega.$ Suppose that $m\geq N,$ consider the perturbation $\beta_m=(\phi_m+\sum_{p\in P}\delta_p)^\circ$ of the maximal stable. We would like to know how the limiting shape looks like.  For example, let $C_m$ be the subset of $V(\Gamma[m])\cap\Omega$ where $\beta_m$ is not maximal. What is the Hausdorff limit  (if it exists)  of $C_m$ in $\Omega?$ 

\subsection{Boundaries}\label{sec_boundaries} Here we propose a hyperbolic analog of the lattice polygons in $\ZZ^2$ with sides of rational slope.
Let $G$ be a subgroup of $Aut(\mathbb{H}^2)$ consisting of all symmetries of the embedding of $\Gamma$ to $\mathbb{H}^2.$ A geodesic $l\subset\mathbb{H}^2$ is called $\Gamma$-rational if there exist a point $p\in l$ and an element $A\in\G$ such that $p\neq Ap$ and $l$ passes through $Ap.$ 

It might be useful to restrict to the specific types of $\Omega.$ For example in \cite{us} the case of polytope with rational edges plays an important role and in this case the results become simpler. Natural analogues of such sets would be finite intersections of hyperbolic half-planes whose boundaries are $\Gamma$-rational geodesics.

\subsection{Abstract graphs}
Finally we would like to raise the main question of our paper in the case of general symmetric hyperbolic graphs. Let $G$ be a finetely generated group and $S$ be a chosen finite set of its generators. For a subgroup $H$ of $G$ we construct a graph $\Gamma$ with the set of vertices equal to $G/H.$ Two cosets $\alpha,\beta\in G/H$ are connected with an edge in this graph if the exist $s\in S$ such that $\beta=s\alpha.$ This graph is known as {\it relative Cayley graph} (see \cite{kapovich2002geometry}). Scaling procedure and polygonal-type boundaries can be defined as in Sections~\ref{sec_scaling}, \ref{sec_boundaries}.

Let $L\colon V(\Gamma)\rightarrow\mathbb{Z}_{\geq 0}$ be the function measuring the distance from $H\in\Gamma,$ i.e. $L(\alpha)=\mathrm{dist}_{\Gamma}(\alpha,H).$ Consider $\Gamma_m=L^{-1}([0,m]),$ a finite subset $P$ of $\Gamma_m$ and $\phi_m$ be the maximal stable state on $\Gamma_m.$ Is it possible to describe the relaxation for $\phi_m+\sum_{p\in P}\delta_p$ in the case of hyperbolic $\Gamma$ (see \cite{kapovich2002geometry})?
  
\bibliography{../sandbib}
\bibliographystyle{abbrv}

\end{document}